\documentclass[11pt]{amsart}
\usepackage{amssymb}
\usepackage{bm}
\usepackage{graphicx}
\usepackage[centertags]{amsmath}
\usepackage{amsfonts}
\usepackage{amsthm}
\usepackage{a4}
\usepackage{latexsym}
\usepackage{amsmath}
\usepackage{amsfonts}
\usepackage{amssymb}
\usepackage{amscd}
\usepackage{amsthm}
\usepackage{layout}
\usepackage{color}
\usepackage{tikz}

\theoremstyle{plain}
\theoremstyle{definition}
\newtheorem{theorem}{Theorem}[section]

\newtheorem{corollary}[theorem]{Corollary}
\newtheorem{definition}[theorem]{Definition}
\newtheorem{remark}[theorem]{Remark}

\newtheorem*{thm*}{Theorem}

\newtheorem*{dfn*}{Definition}

\newtheorem*{cor*}{Corollary}

\newtheorem*{prp*}{Proposition}

\newtheorem*{rmk*}{Remark}

\newcommand{\vertiii}[1]{{\left\vert\kern-0.25ex\left\vert\kern-0.25ex\left\vert #1 \right\vert\kern-0.25ex\right\vert\kern-0.25ex\right\vert}}

\newtheorem{defn}{Definition}[section]
\newtheorem{qn}[defn]{Question}

\theoremstyle{remark}

\theoremstyle{plain}
\newtheorem{thm}[defn]{Theorem}
\newtheorem{lem}[defn]{Lemma}
\newtheorem{cor}[defn]{Corollary}

\renewcommand{\thefootnote}{\arabic{footnote}}

\long\def\symbolfootnote[#1]#2{\begingroup%
\def\thefootnote{\fnsymbol{footnote}}\footnote[#1]{#2}\endgroup}

\begin{document}
\title[Ball covering property]{Ball covering property on operators and Calkin algebras}
\author[S. Siju]{Sreejith Siju}
\author[B. Zheng]{Bentuo Zheng}
\address{sreejithsiju5@gmail.com, ssiju@memphis.edu, Department of Mathematical sciences, University of Memphis, TN, USA, 38152}
\address{bzheng@memphis.edu, Department of Mathematical sciences, University of Memphis, TN, USA, 38152}

\date{}

\maketitle

\symbolfootnote[0]{}

\begin{abstract}
A Banach space $X$ is said to have the ball covering property (BCP) if the unit sphere of $X$ can be covered by countably many open balls $B(x_i, r_i)$ with $r_i\leq \|x_i\|$ for each $i\in\mathbb{N}$. If there are $R, \delta>0$ so that $r_i\leq R$ and $\|x_i\|-r_i>\delta$ for all $i\in\mathbb{N}$, then we say that $X$ has the uniform ball covering property (UBCP). In this paper, we show that if $X$ has an $1$-unconditional basis or $X$ is an $1$-complemented subspace of a Banach space with a shrinking $1$-unconditional basis, then the Calkin algebra $\mathcal{B}(X)/\mathcal{K}(X)$ fails the BCP. It is also shown that if $X$ has a shrinking unconditional basis with unconditional constant less than 2, then $\mathcal{B}(X)$ has the UBCP. 
\end{abstract}

\section{Introduction}
A normed space $X$ is said to have the {\em ball-covering property} (BCP) \cite{C} if its unit sphere can be covered by countably many closed, or equivalently, open balls $(B_n)$, none of which contains the origin of $X$. The centers of the balls are called the BCP points. If in addition, the radii of the covering balls are bounded, then we say $X$ has the {\em strong ball-covering property} (SBCP) \cite{LZh2}. If $X$ has the SBCP and there exists an open ball $B_0$ containing the origin such that $B_n\cap B_0=\phi$ for all $n\in\mathbb{N}$, then $X$ is said to have the {\em uniform ball-covering property} (UBCP) \cite{LZh2}. It follows directly that the UBCP implies the SBCP and the SBCP implies the BCP. But in general, these properties are different \cite{LZh2}.

Separable normed spaces automatically have the BCP. It was shown in \cite{CCL} that BCP is not preserved under linear isomorphisms by providing renormings of $\ell_{\infty}$. Actually it is easy to see that $\ell_{\infty}$ has the BCP while $L_{\infty}$ can be shown to fail the BCP \cite{LZh1}. Although the BCP is a geometric property, it is closely related to the topological property of the dual spaces. To be more precise, V. P. Fonf and C. Zanco \cite{FZ} proved that $X^*$ is $w^*$-separable if and only if $X$ can be ($1+\varepsilon$)-equivalently renormed to have the SBCP for all $\varepsilon>0$. The BCP is also connected with other commonly used properties, such as the Radon-Nikodym property, uniform convexity and dentability \cite{CWWZ, SC2, SC3}.

There are quite a few recent advances about the BCP. It is proved by Z. Luo and B. Zheng \cite{LZh1} that $\left(\sum\oplus X_k\right)_E$ has the BCP if and only if each summand $X_k$ has the BCP, where $E$ is either a Lorentz sequence space, a separable Orlicz sequence space, or $\ell_\infty$. It was also shown in \cite{LZh1} that if $(\Omega,\Sigma,\mu)$ is a separable measure space, then the space of Bochner integrable functions $L^p(\mu,X)$ has the BCP if and only if $X$ has the BCP. In general, the BCP does not pass to subspaces and quotients \cite{CCL}. If 
$X$ is the $\ell_p$ direct sum of $Y$ and $Z$, then $X$ has the BCP if and only if $Y$ and $Z$ have the BCP. However, in \cite{LLLZ}, M. Liu, R. Liu, J. Lu and B. Zheng proved that $\mathcal{B}(H)$ has the BCP and $L_{\infty}$ is an $1$-complemented subspace of $\mathcal{B}(H)$ which fails the BCP. 

It is known that the BCP does not pass to quotient spaces. The first such example is $\ell_{\infty}/c_0$ which was given by L. Cheng, Q. Cheng and X. Liu \cite{CCL}. Recently, Q. Bao, R. Liu and J. Shen showed that if $X$ has a shrinking $1$-unconditional basis, then the Calkin algebra $\mathcal{B}(X)/\mathcal{K}(X)$ fails the BCP. However, the classical case $\mathcal{B}(\ell_1)/\mathcal{K}(\ell_1)$ is not included. A more general question is whether the shrinkingness can be removed from the assumption. \\

\noindent
{\bf Question 1:} {\it Let $X$ be a Banach space with an $1$-unconditional basis. Does the Calkin algebra $\mathcal{B}(X)/\mathcal{K}(X)$ fail the BCP}?\\

In Section 2, we will give an affirmative answer to Question 1.  Moreover, by using the machinery of the Markushevich basis, we also proved that if $X$ is an $1$-complemented subspace of a Banach space with a shrinking $1$-unconditional basis, then the Calkin algebra $\mathcal{B}(X)/\mathcal{K}(X)$ fails the BCP. It should be pointed out that the unconditionality requirement in Question 1 is essential and can not be replaced by shrinking Schauder basis. For example, if we take $X$ to be the Argyros-Haydon space \cite{AH}, then every bounded linear operator on $X$ is a scalar multiple of the identity plus a compact operator. An immediate observation is that $\mathcal{B}(X), \mathcal{K}(X)$ and $\mathcal{B}(X)/\mathcal{K}(X)$ all have the BCP. P. Motakis \cite{M1} proved that for every compact metric space $K$ there exists a Banach space $X_K$ whose Calkin algebra is isometric to $C(K)$. It is obvious that such $C(K)$ spaces are separable and hence have the BCP.  

It was proved in \cite{LLLZ} that $\mathcal{B}(c_0)$ and $\mathcal{B}(\ell_p)$ $(1\leq p<\infty)$ have the UBCP. All these spaces have $1$-unconditional bases. A natural question is the following: \\

\noindent
{\bf Question 2:} {\it Let $X$ be a Banach space with an unconditional basis. Does the space $\mathcal{B}(X)$ have the UBCP}?\\

In \cite{BLS}, it was shown that if $X$ has a monotone shrinking unconditional basis with unconditional constant less than $2$, then $\mathcal{B}(X)$ has the BCP. In Section 3, we prove that the monotonicity is redundant and $\mathcal{B}(X)$ has the UBCP. Actually, we show that if $X$ is a Banach space with $X^*$ separable and $Y$ has an unconditional basis with unconditional constant less than $2$, then the space $\mathcal{B}(X, Y)$ has the UBCP.

\section{BCP of Calkin Algebra}
 Let $X$ be a Banach space with a normalized Schauder basis $(e_n)_n$. Then for every $x \in X$ can be written uniquely as $x =\sum_n e_n^*(x) e_n$, where $(e_n^*)_n$ denotes the sequence of 
 biorthogonal functionals for $(e_n)_n$. As usual, we denote by 
 $$\hbox{supp } x = \{n \in \mathbb{N}: \ e_n^*(x) \not =0\}.$$
 We consider $c_{00}(X)$ the space of all elements of $X$ with finite supports. For a subset $A$ of $X$, $[A]$ denotes the linear span of $A$. For $x \in c_{00}(X)$, 
we use $m(x, (e_i))$ and $M(x, (e_i))$ to denote respectively the minimum and maximum of $\hbox{supp } x$ with respect to the basis $(e_i)$. The set $\{i\in\mathbb{N}: m(x, (e_i))\leq i\leq M(x, (e_i))\}$ is denoted by $I(x, (e_i))$. When there is no confusion about the basis, we simply write $m(x),M(x)$ and $I(x)$ respectively. We say that a sequence  $(x_j)_j$ in $c_{00}(X)$ is a block sequence of  $(e_n)_n$ if there exist $0 = q_1 < q_2 <\cdots$ such
that, for all $j \geq 1$, $x_j$ is in $[e_i :\  q_{j-1} < i \leq q_j]$. 

Let $T$ be a bounded linear operator from a Banach space $X$ into a Banach space $Y$. The essential norm of $T$ is defined by
$$
\|T\|_e=\inf\{\|T-K\|: K\in \mathcal{K}(X, Y)\}.
$$

\begin{definition}
A Banach space $X$ is said to have property (E) if for every bounded linear operator $T\in \mathcal{B}(X)$, there exists a normalized weakly null sequence $(x_n)\subset X$ such that
$$
\|T\|_e = \lim \|Tx_k\|.
$$
\end{definition}

\begin{lem}\label{lemm}
    Let $X$ be an $1$-complemented subspace of a Banach space $Y$ with a bimonotone shrinking Schauder basis $(e_n)_n$. Then $X$ has property (E).
    \end{lem}

\begin{proof}
For each $n\in\mathbb{N}$, let $P_n$ be the natural projection from $Y$ onto $[e_i: 1\leq i\leq n]$. Since $(e_n)_n$ is bimonotone, $\|P_n\|=\|I-P_n\|=1$. Let $P$ be a norm one projection from $Y$ onto $X$ and let $T$ be any bounded linear operator in $\mathcal{B}(X)$. If $T$ is compact, then the conclusion is obvious. So we may assume that $T$ is not compact and $\|T\|_e>0$. Fix a sequence of positive real numbers $(\varepsilon_n)$ which converges to $0$. Since $\|T\|_e = \|T-TPP_n|_X)\|_e=\|TP-TPP_n|_X\|_e\leq \|TP(I-P_n)\|$, we can find a normalized sequence $(y_n)_n$ in $X$ such that for each $n\in\mathbb{N}$,
$$
\|T\|_e<\|TP(I-P_n)y_n\|+\epsilon_n.
$$
Let $x_n=P(I-P_n)y_n/\|P(I-P_n)y_n\|$. Then $(x_n)$ is a normalized weakly null sequence in $X$ and
$$
\|T\|_e\leq\limsup  \|Tx_n\|.
$$
On the other hand, for any compact operator $L$ on $X$, we have 
\begin{equation*}
    \limsup\|Tx_n\| = \limsup\|(T-L)x_n\|  \leq \|T-L\|. 
\end{equation*}
By taking infimum on the right hand side, we get
\begin{equation}\label{12}
    \limsup\|Tx_n\| \leq \|T\|_e.
\end{equation}
Combining the results above, we have
\begin{equation*}
 \|T\|_e = \limsup \|Tx_n\|.
\end{equation*}
The conclusion follows by passing to a subsequence of $(x_n)$.
\end{proof}

Let $(e_n)$ be a Schauder basis for a Banach space $X$. We say that $(e_n)$ is unconditional if there exists a constant $C\geq 1$ so that for any sequences $(a_n)$ and $(b_n)$ with $|a_n|\leq |b_n|, \forall n\in\mathbb{N}$,
$$
\Big\|\sum a_n e_n\Big\|\leq C \Big\|\sum b_n e_n\Big\|.
$$
The best constant $C$ is called the unconditional constant for $(e_n)$.

Recall that an M-basis for a separable Banach space $X$ is a biorthogonal system $(x_i, x_i^*)\subset X\times X^*$ such that $[x_i]=X$ and $\overline{[x_i^*]}^*=X^*$. $[x_i]$ denotes the closed linear span of the $x_i$'s and $\overline{[x_i^*]}^*$ is the weak$^*$ closure of the linear span of the $x_i^*$'s. Markushevich \cite{M1} (or see \cite[Proposition 1.f.3]{LT}) proved that every separable Banach space admits an M-basis. If $X^*$ is separable, we can choose an M-basis $(x_i, x_i^*)$ so that $[x_i^*]=X^*$. In such a case, we call $(x_i, x_i^*)$ a shrinking M-basis for $X$. If $(x_i, x_i^*)$ is a shrinking M-basis for $X$, then every normalized block sequence $(y_i)$ of $(x_i)$ converges weakly to $0$. Given an M-basis $(x_i, x_i^*)$ and $1=n_1<n_2<...$, the sequence $(X_i, X_i^*)$ defined by $X_i=[x_j]_{j=n_i}^{n_{i+1}-1}, X_i^*=[x_j^*]_{j=n_i}^{n_{i+1}-1}$ is called a blocking of $(x_i, x_i^*)$. 

\begin{thm}\label{thm1}
    Let $X$ be an infinite dimensional complemented subspace of a Banach space $Y$ with a shrinking unconditional basis $\{e_n\}$. If $X$ has property (E), then the Calkin algebra $\mathcal{B}(X)/\mathcal{K}(X)$ does not have the ball covering property.   
\end{thm}

\begin{proof}
 Assume that the sequence $(K_n)_n$ is a ball covering of $\mathcal{B}(X)/\mathcal{K}(X)$. We will prove that there exists an operator $T\in \mathcal{B}(X)$ such that, 
\begin{enumerate}
    \item $\|T\|_e=1$ and
    \item $\|K_n-T\|_e \geq\|K_n\|_e$ for all $n$. 
\end{enumerate}
Let $(f_i, f_i^*)$ be a shrinking M-basis for $X$ so that $\|f_i\|=1$ for all $i\in\mathbb{N}$.
Since $X$ has property (E), by a standard perturbation argument and induction, we can select for each $n$ a normalized block sequence $(y_k^n)_k$ of $(f_i)$ such that
\begin{enumerate}
    \item $\|K_n\|_e = \lim_k \|K_ny_k^n\|$,
    \item $I(y_k^n, (f_i))\cap I(y_j^m, (f_i))=\phi$ whenever $n\neq m$.
    \end{enumerate}
By (1) and (2), we can relabel the set $\{y_k^n\}_{k, n}$ as a block sequence $(y_i)$ of $(f_i)$ so that 
$$
\|K_n\|_e = \limsup_m \|K_ny_m\| \mbox{ for all } n.
$$

Let $(\epsilon_i)$ be a sequence of positive real numbers decreasing to $0$. Since both $(e_i)$ and $(f_i)$ are shrinking, again by induction, we are able to find a lacunary subsequence $(x_i)$ of $(y_i)$ so that
\begin{enumerate}
    \item[(i)] $\|K_n\|_e = \limsup_m \|K_nx_m\| \mbox{ for all } n$,
    \item[(ii)] There are increasing sequences of natural numbers $(m_j)$ and $(M_j)$ so that $m_j<M_j<m_{j+1}, \forall j$ and
      $$
     \Big\|(\sum_{i\notin [m_n, M_n]} e_i^*\otimes e_i)|_{[f_i]_{i \in I(x_n, (f_i))}}\Big\|<\epsilon_n\mbox{ for all } n,
    $$
    \item[(iii)] For each $j\in\mathbb{N}$, there is a $t_j\in\mathbb{N}$ so that $M(x_j, (f_i))<t_j<m(x_{j+1}, (f_i))$ and
    $$
     \Big\|(\sum_{i\notin (M_j, m_{j+1})} e_i^*\otimes e_i)(f_{t_j})\Big\|<\epsilon_j.
     $$
        \end{enumerate}
    
Now we define an operator $S$ on $X$ as follows,
$$
S=P\circ(\sum_{i\in \cup_j (M_j, m_{j+1})} e_i^*\otimes e_i),
$$
where $P$ is a bounded linear projection from $Y$ onto $X$.

Since $(e_i)$ is unconditional, $S$ is bounded. Let $C$ be the unconditional constant for $(e_i)$. By (ii),
$$
\|S(x_j)\|<C\|P\|\epsilon_j\mbox{ for all } j.
$$
So we have $\lim_j\|S(x_j)\|=0$. Next we show that $S$ is not compact. Actually by (iii), we have
\begin{align*}
\|S(f_{t_j})\| 
&=\|P\circ(I-(\sum_{i\notin (M_j, m_{j+1})} e_i^*\otimes e_i))(f_{t_j})\| \\
&=\|f_{t_j}-P\circ(\sum_{i\notin (M_j, m_{j+1})} e_i^*\otimes e_i))(f_{t_j})\| \\
&\geq 1-\|P\|\epsilon_j.
\end{align*}
Therefore $\limsup_j\|S(f_{t_j})\|\geq1$. Since $(f_{t_j})_j$ is weakly null, $S$ is not compact and $\|S\|_e>0$.

Let $T=S/\|S\|_e$. Hence, for each $n$ 
\begin{align*}
\left\|K_n-T\right\|_e 
&=\left\|K_n-\frac{S}{\|S\|_e}\right\|_e\geq \limsup_j\left\|K_nx_j-\frac{S}{\|S\|_e}x_j\right\| \\
&= \limsup_j\|K_nx_j\| = \|K_n\|_e.
\end{align*}

This finishes the proof.
\end{proof}

By Lemma \ref{lemm} and Theorem \ref{thm1}, we have the following immediate corollary.
\begin{corollary}
Let $X$ be an $1$-complemented subspace of a Banach space $Y$ with a shrinking $1$-unconditional basis. Then the Calkin algebra $\mathcal{B}(X)/\mathcal{K}(X)$ fails the BCP.
\end{corollary}

When we do not have shrinkingness for the unconditional basis, situation is different. 
Next we will show the general case. That is, if $X$ has an $1$-unconditional basis, then the Calkin algebra $\mathcal{B}(X)/\mathcal{K}(X)$ does not have the BCP. The proof of the lemma below is standard which will be omitted.

\begin{lem}\label{lem2}
Let $X$ be a Banach space with a Schauder basis $(e_n)$ and $(P_n)$ be the natural basis projection from $X$ onto $[e_i]_{i=1}^n$. Let $K$ be any compact operator in $\mathcal{B}(X)$. Then
$$
\lim_n \|(I-P_n)K\|=0.
$$
\end{lem}

\begin{lem}\label{lem3}
Let $X$ be a Banach space with a bimonotone basis $(e_n)$ and $(P_n)$ be the natural basis projection from $X$ onto $[e_i]_{i=1}^n$. Then for every operator $T\in\mathcal{B}(X)$, there exist a normalized block sequence $(x_n)$ of $(e_n)$ and an increasing sequence of natural numbers $(k_n)$ so that
$$
\|T\|_e=\lim_n\|(I-P_{k_n})T x_n\|.
$$
\end{lem}

\begin{proof}
By the definition of the essential norm, for any $n\in\mathbb{N}$, there exists a compact operator $S_n$ so that
$$
\|T\|_e>\|T-S_n\|-1/n.
$$
By Lemma \ref{lem2}, there exists a $t_n\in\mathbb{N}$ such that
$$
\|(I-P_{t_n})S_n\|<1/n.
$$
So we have
$$
\|T\|_e>\|T-S_n\|-1/n\geq \|(I-P_{t_n})(T-S_n)\|-1/n\geq \|(I-P_{t_n})T\|-2/n.
$$
On the other hand,
$$
\|T\|_e=\|(I-P_{t_n})T\|_e\leq \|(I-P_{t_n})T\|.
$$
Therefore, we get
$$
\|(I-P_{t_n})T\|-2/n\leq\|T\|_e\leq \|(I-P_{t_n})T\|.
$$
Then we can choose a finitely supported norm one vector $y_n$ in $X$ so that 
$$
\|(I-P_{t_n})Ty_n\|>\|(I-P_{t_n})T\|-1/n
$$ and
$$
\|(I-P_{t_n})Ty_n\|-2/n\leq\|T\|_e\leq \|(I-P_{t_n})Ty_n\|+1/n.
$$
By taking limit over $n$, we have
$$
\|T\|_e= \lim_n\|(I-P_{t_n})Ty_n\|= \lim_n\|(I-P_{t_n})T\|.
$$
Now we will show that a subsequence of $(y_n)$ is a small perturbation of a block basis.

Let $r_1=1$. Since $y_1$ is finitely supported, $M(y_1)$ is finite. By Lemma \ref{lem2},
there is $r_2>r_1$ so that
$$
\|(I-P_{t_{r_2}})TP_{M(y_1)}\|<1/2.
$$
In general, since $y_{r_n}$ is finitely supported, we can find $r_{n+1}>r_n$ so that
$$
\|(I-P_{t_{r_{n+1}}})TP_{M(y_{r_n})}\|<1/n.
$$
So we have
\begin{align*}
\|(I-P_{t_{r_{n+1}}})Ty_{r_{n+1}}\|
&=\|(I-P_{t_{r_{n+1}}})TP_{M(y_{r_n})}y_{r_{n+1}}+(I-P_{t_{r_{n+1}}})T(I-P_{M(y_{r_n})})y_{r_{n+1}}\|\\
&<\|(I-P_{t_{r_{n+1}}})T(I-P_{M(y_{r_n})})y_{r_{n+1}}\|+1/n\\
\end{align*}
Let $x_1=y_1, x_{n+1}=(I-P_{M(y_{r_n})})y_{r_{n+1}}/\|(I-P_{M(y_{r_n})})y_{r_{n+1}}\|$ and $k_n=t_{r_n}$. Then $(x_n)$
is a normalized block basis and 
$$
     \|(I-P_{t_{r_{n+1}}})Ty_{r_{n+1}}\|-1/n < \|(I-P_{k_{n+1}})Tx_{n+1}\|\leq \|(I-P_{t_{r_{n+1}}})T\|.
$$
Hence
$$
\|T\|_e= \lim_n\|(I-P_{k_n})Tx_n\|= \lim_n\|(I-P_{k_n})T\|.
$$
\end{proof}

\begin{thm}\label{thm2}
Let $X$ be a Banach space with an $1$-unconditional basis. Then the Calkin algebra $\mathcal{B}(X)/\mathcal{K}(X)$ does not have the BCP.
\end{thm}

\begin{proof}
We assume that there is a ball covering $(K_n)_n$ of $\mathcal{B}(X)/\mathcal{K}(X)$. To reach a contradiction, we will construct an operator $T\in \mathcal{B}(X)$ such that
\begin{enumerate}
    \item $\|T\|_e=1$ and
    \item $\|K_n-T\|_e \geq\|K_n\|_e$ for all $n$. 
\end{enumerate}
For each $K_n$, by Lemma \ref{lem3}, there exist a normalized block sequence $(y_i^n)$ of $(e_n)$ and an increasing sequence of natural numbers $(k_i^n)$ so that
$$
\|K_n\|_e=\lim_i\|(I-P_{k_i^n})K_n y_i^n\|.
$$
Inductively, we can select a block sequence $(x_i)$ such that
\begin{enumerate}
    \item[(i)] $(x_i)$ contains a subsequence of $(y_i^n)$ for each $n$ and hence 
    $$
    \|K_n\|_e \leq \limsup_m \|(I-P_{k_m})K_nx_m\| \mbox{ for all } n,
    $$
    \item[(ii)] For each $i$,
      $$
     M(x_i)+1<m(x_{i+1}).
    $$
            \end{enumerate}
Let $S\in\mathcal{B}(X)$ be defined as
$$
S=\sum_i e_{M(x_i)+1}^*\otimes e_{M(x_i)+1}.
$$
Since $(e_i)$ is unconditional, $S$ is bounded. Since $S(e_{M(x_i)+1})=e_{M(x_i)+1}$ for each $i$, $S$ is not compact. Let $T=S/\|S\|_e$. Then $\|T\|_e=1$ and $T(x_i)=0$ for each $i$. We will show $\|K_n-T\|_e\geq \|K_n\|_e$ for each $n$. For any $\epsilon>0$ and $n\in\mathbb{N}$, there exists a compact operator $K_\epsilon$ so that 
$$\|K_n-T\|_e>\|K_n-T- K_\epsilon\|-\epsilon.
$$
Then by Lemma \ref{lem2}, there exists an $N_\epsilon\in\mathbb{N}$ such that if $m>N_\epsilon$, then
$$
\|(I-P_m)K_\epsilon\|<\epsilon.
$$
So we have for any $m>N_\epsilon$,
\begin{align*}
\|K_n-T\|_e
&>\|K_n-T- K_\epsilon\|-\epsilon\\
&\geq\|(I-P_{k_m})(K_n-T -K_\epsilon)\|-\epsilon\\
&\geq\|(I-P_{k_m})(K_n-T)\|-2\epsilon\\
&\geq\|(I-P_{k_m})(K_n-T)x_m\|-2\epsilon\\
&\geq\|(I-P_{k_m})K_nx_m\|-2\epsilon.
\end{align*}
Taking limit over $m$, we get
$$
\|K_n-T\|_e\geq\limsup_m\|(I-P_{k_m})K_nx_m\|-2\epsilon=\|K_n\|_e-2\epsilon.
$$
Since $\epsilon$ is arbitrary, we have
$$
\|K_n-T\|_e\geq\|K_n\|_e.
$$
This concludes the proof.
\end{proof}

\begin{remark}
Note that the theorem above also holds if $X$  has only an unconditional finite dimensional decomposition (FDD). The unconditionality assumption in the theorem cannot be removed without adding other assumptions. Indeed, the space $\mathfrak{X}_{AH}$, introduced by Argyros and Haydon \cite{AH}, has a shrinking Schauder basis but its Calkin algebra is one-dimensional, which of course has the BCP.
\end{remark}

However, we do not know any example of non-separable Banach space $X$ whose Calkin algebra $\mathcal{B}(X)/\mathcal{K}(X)$ has the BCP. Here is an interesting open problem:

\begin{qn}
Does there exist a non-separable Banach space $X$ whose Calkin algebra $\mathcal{B}(X)/\mathcal{K}(X)$ has the BCP?
\end{qn}

\section{BCP for $\mathcal{B}(X)$}
As mentioned in the introduction, it is known that $\mathcal{B}(c_0)$ and $\mathcal{B}(\ell_p)$ $(1\leq p<\infty)$ has the BCP. The spaces $c_0$ and $\ell_p$ $(1\leq p<\infty)$ have $1$-unconditional bases. In this section, we will prove that $\mathcal{B}(X)$ has the BCP if $X$ has a shrinking unconditional basis with unconditional constant less than $2$. Actually, we have the following more general result.

\begin{thm}
    Let $X$ and $Y$ be Banach spaces such that $X^*$ is separable and $Y$ has an unconditional basis with unconditional constant less than $2$.  Then any subspace $E$ of $\mathcal{B}(X, Y)$ containing $\mathcal{F}(X, Y)$ has the UBCP. 
\end{thm}
\begin{proof}
It is sufficient to show that the UBCP points of $\mathcal{B}\left(X, Y\right)$ can be taken in $\mathcal{F}\left(X, Y \right)$. Let $\left(e_n\right)$ be a normalized $\gamma$-unconditional basis of $Y$ with $1\leq\gamma< 2$  and $\left(e_n^*\right)$ be the corresponding biorthogonal functionals in $Y^*$. Since $X^*$ is separable, there exists a countable dense subset $\left(x_n^*\right)$ in $B(0, \gamma)$ of $X^*$.
  We will show that the countable set $\mathcal{A}$, given by
$$
\mathcal{A}=\left\{\lambda\sum_{i\in F}x_{m_i}^* \otimes e_i : \lambda\in \mathbb{Q}\cap[0, 2\gamma], F\subset \mathbb{N} \mbox{ finite}, m_i\in\mathbb{N} \right\}
$$
is a set of UBCP points for $\mathcal{B}\left(X, Y\right)$.

Let $T=\sum_{n=1}^{\infty} e_n^* T \otimes e_n \in S_{\mathcal{B}\left(X, Y\right)}$, where the series converges in the strong operator topology. Let $$\beta_1 = \left\{\sup_F\left\|\sum_{i\in F} {e_i^* T}\otimes e_i\right\|: F\subset \mathbb{N} \mbox{ finite}\right\},$$
and
$$\beta = \left\{\sup_F\left\|\sum_{i\in F} {\epsilon_ie_i^* T}\otimes e_i\right\|: F\subset \mathbb{N} \mbox{ finite}, \epsilon_i=\pm1\right\}.$$

Since $\|T\|=1$ and $(e_i)$ is $\gamma$-unconditional, by the definition of $\beta_1, \beta$ and triangle inequality, it is easy to see that
 $$
 \beta\leq\gamma<2\mbox{ and }1\leq \beta_1\leq \beta\leq 2\beta_1.
 $$ 

\textbf{Case 1:} Assume $\beta <\frac{7}{4}\beta_1$.

We can choose a $\lambda\in \mathbb{Q}$ and $\frac{1}{4}>\delta>0$ such that $2\beta_1>\lambda>\frac{\beta}{1-2\delta}+\frac{1}{8}$, which is possible since $\frac{1}{8}+\beta<\frac{1}{4}+\frac{7}{4}\beta_1\leq \frac{\beta_1}{4}+\frac{7}{4}\beta_1 =2\beta_1$. 

Pick a finite subset $F\subset \mathbb{N}$ such that
$$
{\lambda}<2{\theta},
$$
where  $\theta=\left\|\sum_{i\in F} {e_i^* T}\otimes e_i\right\|\leq\beta_1$. Since $\left\|\frac{{e_i^* T}}{\theta}\right\| \leq \gamma$ for all $i\in F$,  we can find $m_i \in \mathbb{N}$ such that
$$
\sum_{i\in F} \left\|\left(x_{m_i}^*-\frac{{e_i^* T}}{\theta}\right)\right\|<{\delta}.
$$
Then
\begin{align*}
    \left\|\lambda\sum_{i\in F} x_{m_i}^*\otimes e_i\right\| & = \left\|\lambda\sum_{i\in F} \left(x_{m_i}^*\otimes e_i - \frac{{e_i^* T}}{\theta} \otimes e_i+ \frac{{e_i^* T}}{\theta}\otimes e_i\right) \right\|\\
    &\geq \lambda\left\|\sum_{i\in F}\frac{{ e_i^* T}}{\theta}\otimes e_i\right\|-\lambda\left\|\sum_{i\in F} \left(x_{m_i}^*\otimes e_i - \frac{{e_i^* T}}{\theta}\otimes e_i\right)\right\|\\
   & \geq \lambda(1-{\delta}).
\end{align*}

Also, since $|\frac{\lambda}{\theta}-1| = \frac{\lambda}{\theta}-1\leq1$, we get,
\begin{align*}
   \left \|T-\sum_{i\in F}\lambda x_{m_i}^* \otimes e_i\right\|
   &= \left\| \sum_{i=1}^{\infty} e_i^* T \otimes e_i -  \sum_{i\in F}\lambda x_{m_i}^* \otimes e_i\right\| \\
   & \leq  \left\|\left(1-\frac{\lambda}{\theta}\right) \sum_{i\in F} e_i^* T \otimes e_i+\sum_{i\in F^c} e_i^* T \otimes e_i\right\|
   +\sum_{i\in F}\left\|\lambda\left(\frac{e_i^* T}{\theta}-x_{m_i}^*\right) \otimes e_i\right\| \\
   & \leq \left\| \sum_{i\in F}\left( 1-\frac{\lambda}{\theta}\right)e_i^* T\otimes e_i+\sum_{i\in F^c} e_i^* T\otimes e_i\right\|+{\lambda\delta}\\
   & \leq \left(\frac{\lambda}{\theta}-1\right) \left\| -\sum_{i\in F}e_i^* T\otimes e_i+\sum_{i\in F^c} e_i^* T\otimes e_i\right\|+\left(2-\frac{\lambda}{\theta}\right) \left\|\sum_{i\in F^c} e_i^* T\otimes e_i\right\|+\lambda\delta \\
   &\leq \beta+ \lambda\delta\\
    & \leq  (\lambda-\frac{1}{8})(1-2\delta)+\lambda\delta\\
    &\leq \lambda(1-\delta)-\frac{1}{8}(1-2\delta)\\
    & < \lambda(1-\delta)-\frac{1}{8}\\
     & < \left\|\lambda\sum_{i\in F} x_{m_i}^*\otimes e_i\right\|-\frac{1}{8}.
    \end{align*}
    
\textbf{Case 2:} Assume $\beta \geq\frac{7}{4}\beta_1$. Let $\delta=2-\gamma$. By the definition of $\beta_1$, we choose a finite subset $F\subset\mathbb{N}$ so that
$$
\left\|\sum_{i\in F} e_i^*T\otimes e_i\right\|>\beta_1-\frac{\delta}{8}.
$$
Since $\left\|e_i^*T\right\|\leq\gamma$ for each $i\in F$, we can find $(x_{m_i}^*)_{i\in F}$ so that
$$
\sum_{i\in F} \left\|{e_i^* T}-x_{m_i}^*\right\|<\frac{\delta}{8}.
$$
Then
\begin{align*}
    \left\|\sum_{i\in F} 2x_{m_i}^*\otimes e_i\right\| & > 2(\beta_1-\frac{\delta}{8})-2\times\frac{\delta}{8}=2\beta_1-\frac{\delta}{2}.
   \end{align*}
Now we will show that $T$ can be covered by the ball centered at $2\sum_{i\in F} x_{m_i}^*\otimes e_i$.

\begin{align*}
\left \|T-2\sum_{i\in F}x_{m_i}^* \otimes e_i\right\|
   &= \left\| \sum_{i=1}^{\infty} e_i^* T \otimes e_i - 2\sum_{i\in F}x_{m_i}^* \otimes e_i\right\| \\
    &\leq \left\| \sum_{i=1}^{\infty} e_i^* T \otimes e_i - 2\sum_{i\in F}e_i^*T \otimes e_i\right\|+2\left\|\sum_{i\in F}x_{m_i}^* \otimes e_i- \sum_{i\in F}e_i^*T \otimes e_i\right\|\\
   &\leq \left\| \sum_{i\in F^c} e_i^* T \otimes e_i-\sum_{i\in F} e_i^* T \otimes e_i \right\| + 2\times\frac{\delta}{8}\\
   &\leq \beta+\frac{\delta}{4}\leq \gamma+\frac{\delta}{4} =2-\frac{3}{4}\delta\\
   &\leq 2\beta_1-\frac{3}{4}\delta\\
   & <  \left\|2\sum_{i\in F} x_{m_i}^*\otimes e_i\right\|-\frac{\delta}{4}.
   \end{align*}

This finishes the proof.

\end{proof}

\begin{cor}
Let $X$ be a Banach space which has a shrinking unconditional basis with unconditional constant less than 2. Then $\mathcal{B}(X)$ has the UBCP.
\end{cor}


\begin{thebibliography}{150}
\bibitem{AH}
S.A. Argyros and R.G. Haydon, A hereditarily indecomposable $\mathcal{L}_\infty$-space that solves the scalar-plus-compact problem, {\em Acta Mathematica}, vol. {\bf 206}, 1-54, (2011).

\bibitem{BLS}
Q. Bao, R. Liu and J. Shen, The ball-covering property of non-commutative spaces of operators on Banach spaces, Banach Journal of Mathematical Analysis 19 (2025), no. 34.

\bibitem{C} 
L. Cheng, Ball-covering property of Banach spaces, Israel J. Math. 156 (2006) 111-123.

\bibitem{CCL} 
L. Cheng, Q. Cheng, X. Liu, Ball-covering property of Banach spaces that is not preserved under linear isomorphisms, Sci. China Ser. A 51 (2008) 143-147.

\bibitem{CWWZ} 
L. Cheng, B. Wang, W. Zhang, Y. Zhou, Some geometric and topological properties of Banach spaces via ball coverings, J. Math. Anal. Appl. 377 (2011) 874-880.

\bibitem{DU}
J. Diestel, J.J.Jr. Uhl; Vector measures. Mathematical Surveys, No. {\bf 15}. American Mathematical Society, Providence, RI, (1977).

\bibitem{FZ} 
V.P. Fonf, C. Zanco, Covering spheres of Banach spaces by balls, Math. Ann. 344 (2009) 939-945.

\bibitem{LLLZ}
M. Liu, R. Liu, J. Lu, B. Zheng; Ball covering property from commutative function spaces to non-commutative spaces of operators. {\em J. Funct. Anal.} {\bf 283} (2022), no. 1, Paper No. 109502, 15 pp.

\bibitem{LT} 
J. Lindenstrauss and L. Tzafriri, {\it Classical Banach Spaces I, Sequence Spaces}, 
Springer-Verlag, New York, 1977.

\bibitem{LZh1} 
Z. Luo, B. Zheng, Stability of ball covering property, Studia Math. 250 (2020) 19-34.

\bibitem{LZh2} 
Z. Luo, B. Zheng, The strong and uniform ball covering properties, J. Math. Anal. Appl. 499 (2021) 125034.

\bibitem{MW}
J. Mach and J. Ward, Approximation by Compact Operators on Certain Banach Spaces, Journal of Approximation Theory 23 (3), (1978), 274-286.

\bibitem{M1}
P. Motakis, Separable spaces of continuous functions as Calkin algebras, Journal of the American Mathematical Society, 37 (1), (2025), 1-37.

\bibitem{SC2} 
S. Shang, Y. Cui, Locally 2-uniform convexity and ball-covering property in Banach space, Banach J. Math. Anal. 9 (2015) 42-53.

\bibitem{SC3} 
S. Shang, Y. Cui, Dentable point and ball-covering property in Banach spaces, J. Convex Anal. 25 (2018) 1045-1058.


\end{thebibliography}
\end{document}